\documentclass[oneside,english]{amsart}

\usepackage{amssymb}
\usepackage{amscd}

\numberwithin{equation}{section} 
\numberwithin{figure}{section} 
  \theoremstyle{plain}
  \newtheorem*{thm*}{Theorem}
  \theoremstyle{plain}
  \newtheorem{thm}{Theorem}[section]
  \theoremstyle{definition}
  \newtheorem{defn}[thm]{Definition}
  \theoremstyle{plain}
  \newtheorem{lem}[thm]{Lemma}
  \theoremstyle{plain}
  
  \theoremstyle{plain}
  
  \theoremstyle{remark}
  \newtheorem{rem}[thm]{Remark}
  \theoremstyle{remark}
  \newtheorem*{acknowledgement*}{Acknowledgement}

\begin{document}

\title[Sprays metrizable by Finsler functions of constant flag curvature]
{Sprays metrizable by Finsler functions of constant flag curvature}

\author[Bucataru]{Ioan Bucataru} \address{Ioan Bucataru, Faculty of
  Mathematics, Alexandru Ioan Cuza University \\ Ia\c si, 
  Romania} \urladdr{http://www.math.uaic.ro/\textasciitilde{}bucataru/}

\author[Muzsnay]{Zolt\'an Muzsnay} \address{Zolt\'an Muzsnay, Institute
  of Mathematics, University of Debrecen \\Debrecen, 
  Hungary} \urladdr{http://www.math.klte.hu/\textasciitilde{}muzsnay/}

\date{\today}

\begin{abstract}
In this paper we characterize sprays that are metrizable by Finsler functions
of constant flag curvature. By solving a particular case of the Finsler
metrizability problem we provide the necessary and sufficient
conditions that can be used to decide whether or not a given homogeneous system of second order ordinary differential
equations represents the Euler-Lagrange equations of a Finsler function of constant
flag curvature. The conditions we provide are tensorial equations
on the Jacobi endomorphism. We identify the
class of homogeneous SODE where the Finsler metrizability is equivalent with
the metrizability by a Finsler function of constant flag curvature. 
\end{abstract}



\maketitle

\section{Introduction}

The inverse problem of Lagrangian mechanics can be formulated as
follows: decide whether or not a given system of second order ordinary differential
equations (SODE) coincides with the Euler-Lagrange equations of some
Lagrangian, \cite{AT92, BD09, GM00, Krupkova97, MFLMR90, Sarlet82}. When the given system of SODE is homogeneous and the
Lagrangian to search for is the square of a Finsler function, the
problem is known as the Finsler metrizability problem, \cite{BM12,
  KS85, Muzsnay06, SV02}. If the sought after Lagrangian is a Finsler function, the problem is known as the
projective metrizability problem, or as the Finslerian version of
Hilbert's fourth problem, \cite{Alvarez05, BM11, BM12, Crampin07,
  CMS12, Yang11}.  

In this paper we address the special case of the Finsler
metrizability problem, where the Finsler function we seek for has
constant curvature. When the spray has zero constant curvature, then
there is no obstruction for the existence of a locally defined Finsler
structure that metricizes the given spray, \cite{BM11, Crampin07,
  Muzsnay06}. Therefore, in this work we will
focus on the case when the curvature is non-zero.  

In Theorem \ref{metrizablesc}, we solve the above mentioned problem by
providing a set of equations, which contains an algebraic equation $A)$ and two tensorial differential
equations $D_1)$ and $D_2)$ in \eqref{fmiso}, which have to be satisfied by the Jacobi
endomorphism. One of these two tensorial equations restricts the class of
homogeneous SODE (sprays), which we discuss, to the class
of isotropic sprays. Therefore, we focus our attention on isotropic
sprays and their relation with the Finsler metrizability problem. In
Theorem \ref{metrizables}, we characterize the class of
isotropic sprays $S$ for which the following conditions are
equivalent:
\begin{itemize}
\item $S$ is Finsler metrizable, \item $S$ is metrizable by a Finsler metric of scalar
  flag curvature; 
\item $S$ is Finsler metrizable by an Einstein metric;
\item $S$ is metrizable by a Finsler metric of constant 
  flag curvature; 
\item $S$ is Ricci constant.
\end{itemize}
If the Finsler function is reducible to a Riemannian metric, the equivalence of the above
conditions, on manifolds of dimension grater
than or equal to three, is always true for affine isotropic sprays, and it is not true
in the general Finslerian context. Therefore, Theorem \ref{metrizables}
identifies the class of isotropic sprays where this equivalence is
still true. The last condition is a technical one, can be
checked very easily, and it is common for both Theorems
\ref{metrizablesc} and \ref{metrizables}. Theorem \ref{metrizablesc}
has the advantage of being true for spray spaces of dimension grater
than or equal to two and it has the disadvantage of disregarding
Finsler metrizable sprays of non-constant flag curvature. The main
advantage of Theorem \ref{metrizables} is that it treats sprays for
which the Finsler metrizability is equivalent with the metrizability
by a Finsler metric of constant non-zero flag curvature. The key ingredient in proving Theorem \ref{metrizables} is the
Finslerian version of Schur's Lemma, \cite[Lemma 3.10.2]{BCS00}, which
is true only for spray spaces of dimension grater
than or equal to three. 

Since any spray on a two-dimension manifold is
isotropic, one of the two equations \eqref{fmiso} in Theorem
\ref{metrizablesc} simplify. In Theorem \ref{metrizables2} we
provide necessary and sufficient conditions for the metrizability of a
two-dimensional spray space by a Finsler function of constant (Gaussian) curvature.

To support our results, in the last section, we consider various examples of isotropic sprays that
satisfy, or not, one or more of the necessary and sufficient conditions, which we
provide, for Finsler metrizability. 

\section{Sprays and their geometric setting}

The natural geometric framework for studying systems of second order ordinary
differential equations is the tangent bundle of some configuration
manifold. 

In this work, $M$ denotes a $C^{\infty}$-smooth, real, and
$n$-dimensional manifold. We will denote by $TM$ its tangent
bundle and by $T_0M=TM\setminus\{0\}$ the tangent bundle with the zero
section removed. Local coordinate charts $(U, (x^i))$ on $M$ induce local coordinate
charts $(\pi^{-1}(U), (x^i, y^i))$ on $TM$, where $\pi: TM \to M$ is
the canonical submersion. We will assume that $M$ is a connected
manifold of dimension $n\geq
2$. Therefore, $TM$ and $T_0M$ are $2n$-dimensional connected manifolds.

In this section we discuss the natural geometric setting determined
by a spray $S$, which includes canonical nonlinear connection,
dynamical covariant derivative and curvature tensors. This setting, as
well as the proofs of our results in the next sections, are based on
the Fr\"olicher-Nijenhuis theory and the corresponding differential
calculus that can be developed on $TM$, \cite{GM00, KMS93, Szilasi03}.  There are two
canonical structures on $TM$, which we will use to develop our
setting. One is the tangent structure, $J$, and the other one is the
Liouville vector fiels, $\mathbb{C}$, locally given by 
\begin{eqnarray*} J=\frac{\partial}{\partial y^i} \otimes dx^i, \quad
  \mathbb{C}=y^i\frac{\partial}{\partial y^i}. \end{eqnarray*}

A system of homogeneous second order ordinary differential equations on a manifold
$M$, whose coefficients do not depend explicitly on time, can be
identified with a special vector field on $T_0M$ that is called a spray.  
A vector field $S\in \mathfrak{X}(T_0M)$ is called a \emph{spray} if
$JS=\mathbb{C}$ and $[\mathbb{C}, S]=S$. Locally, a spray $S$ is 
given by 
\begin{eqnarray}
S=y^i\frac{\partial}{\partial x^i} - 2G^i(x,y) \frac{\partial}{\partial
  y^i}, \label{locals}
\end{eqnarray} 
where functions $G^i(x,y)$ are smooth functions on domains of induced
coordinates on $T_0M$ and $2$-homogeneous with respect to the $y$-variable.

It is well known that a spray induces a nonlinear connection, with the
corresponding projectors $h$ and $v$ given by
\begin{eqnarray*}
h=\frac{1}{2}\left(\operatorname{Id}-\mathcal{L}_SJ\right), \quad
v=\frac{1}{2}\left(\operatorname{Id}+\mathcal{L}_SJ\right). 
\end{eqnarray*}
An important geometric structure induced by a spray is the \emph{Jacobi
endomorphism}, which is the vector valued semi-basic $1$-form given by
\begin{eqnarray*}
\Phi=v\circ \mathcal{L}_S h = \mathcal{L}_Sh \circ h.
\end{eqnarray*}
Locally, the Jacobi endomorphism, is given by 
\begin{eqnarray}
\Phi = R^i_j \frac{\partial}{\partial y^i}\otimes dx^j =
\left(2\frac{\partial G^i}{\partial x^j} - S\left(\frac{\partial
      G^i}{\partial y^j}\right) - \frac{\partial
      G^i}{\partial y^r} \frac{\partial
      G^r}{\partial y^j}\right) \frac{\partial}{\partial y^i}\otimes
  dx^j. \label{localphi} \end{eqnarray}
We consider also the \emph{Ricci curvature}, $Ric$, and  the
\emph{Ricci scalar}, $R$, \cite{BR04},
\cite[Def. 8.1.7]{Shen01}, which are given by
\begin{eqnarray}
\operatorname{Ric}=(n-1)R =
R^i_i=\operatorname{Tr}(\Phi). \label{ricr} \end{eqnarray}

For a spray $S$, we consider the map $\nabla: \mathfrak{X}(T_0M) \to
\mathfrak{X}(T_0M)$, given by \cite{BD09}
\begin{eqnarray}
\nabla = \mathcal{L}_S + h\circ \mathcal{L}_Sh + v\circ
\mathcal{L}_Sv. \label{nabla}
\end{eqnarray}  
We require that the action of $\nabla$ on scalar functions is given by
$\nabla f=S(f)$, for $f\in C^{\infty}(T_0M)$. Further requirements that $\nabla$
  satisfies the Leibnitz rule and commutes with contractions allow us to extend its
  action to arbitrary tensor fields on $T_0M$. We will
  refer to $\nabla$ as to the \emph{dynamical covariant derivative} induced by
  the spray $S$. Its action on semi-basic forms was called the 
  \emph{semi-basic derivation} and studied, in connection with the
  inverse problem of the calculus of variation, in \cite{GM00}.

\begin{defn} We say that a spray $S$ is \emph{isotropic} if its Jacobi endomorphism has the form,
\begin{eqnarray}
\Phi = R J - \alpha \otimes \mathbb{C}, \label{isophi}
\end{eqnarray}
where $R \in C^{\infty}(T_0M)$ is the Ricci scalar and $\alpha$ is a semi-basic
$1$-form on $T_0M$. 

A spray $S$ is called \begin{itemize} \item [i)] \emph{Ricci-constant}
  if $d_hR=0$, \item[ii)]  \emph{weakly Ricci-constant} if $S(R)=0.$  \end{itemize}
\end{defn}
It is easy to see that if a spray $S$ is Ricci constant, then it is also weakly Ricci constant.
Indeed, for a Ricci constant spray $S$, it follows that the Ricci scalar satisfies $d_hR=0$. Using the corresponding commutation
formula, we obtain 
$$0= i_Sd_hR = -d_hi_SR + \mathcal{L}_{hS}R + i_{[h, S]}R=S(R), $$ and
hence the spray $S$ is weakly Ricci constant. 

For an isotropic spray $S$, due to the homogeneity condition, it
follows that $0=\Phi(S)=(R - i_S\alpha)\mathbb{C}$ on $T_0M$ and hence
$R=i_S\alpha$. Isotropic sprays can be characterized using the Weyl
curvature, see Prop. 13.4.1 in \cite{Shen01}. The semi-basic vector valued $1$-form $\Phi$ is $2$-homogeneous, which
means $\Phi=[\mathbb{C}, \Phi]$. For an isotropic spray $S$, we have 
\begin{eqnarray*}
\Phi=[\mathbb{C}, \Phi] =[\mathbb{C}, R J] - [\mathbb{C}, \alpha
\otimes \mathbb{C}] = \left(\mathbb{C}(R) - R \right) J -
\mathcal{L}_{\mathbb{C}}\alpha \otimes \mathbb{C}, \end{eqnarray*} 
which implies $\mathbb{C}(R) = 2R$ and
$\mathcal{L}_{\mathbb{C}}\alpha = \alpha$. Therefore, the Ricci
curvature $Ric$ and the Ricci scalar $R$ are $2$-homogeneous, while the $1$-form $\alpha$ is $1$-homogeneous.

\begin{lem} \label{lem:dja} Consider $S$ an isotropic spray, whose Jacobi endomorphism is given
by formula \eqref{isophi}. The following two conditions are
equivalent
\begin{itemize}
\item[i)] $d_J\alpha =0$, \item[ii)] $d_JR  = 2\alpha.$
\end{itemize}
\end{lem}
\begin{proof}
Using the fact that $R=i_S\alpha$ and the commutation rules, it follows
\begin{eqnarray} d_JR=d_Ji_S \alpha = -i_S d_J\alpha + \mathcal{L}_{JS} \alpha +
i_{[J,S]}\alpha = -i_S d_J\alpha + 2\alpha. \label{djr1} \end{eqnarray}
In the above equations, we have used that $[J, S]=h-v$ and since $\alpha$ is a semi-basic $1$-form, it follows
that $i_{[J,S]}\alpha = i_h\alpha =\alpha$. Therefore, the assumption
$d_J\alpha=0$ implies that $d_JR=2\alpha$. The other implication is
straightforward, using the fact that the tangent structure $J$ is
integrable, and hence $d_J^2=0$. \end{proof}
It is known that the projective deformations of a spray preserves the
class of isotropic sprays. These deformations may preserve or not the
condition $d_J\alpha=0$. In subsection \ref{djao} we provide examples of isotropic sprays such
that $d_J\alpha \neq 0$. See formula \eqref{djap} for a convenient
choice of a $1$-homogeneous function $P\in C^{\infty}(T_0M)$. 

\section{Finsler metrizable sprays}

In this section we recall the notion of a Finsler space and its
geodesic spray. We will focus our discussions on Finsler spaces of scalar flag
curvature.

\begin{defn} \label{def:finsler} By a \emph{Finsler function} we mean a
  continuous function $F: TM \to \mathbb{R}$ satisfying the following
  conditions:
  \begin{itemize}
  \item[i)] $F$ is smooth and strictly positive on $T_0M$;
  \item[ii)] $F$ is positively homogeneous of order $1$, which means that
    $F(x,\lambda y)=\lambda F(x,y)$, for all $\lambda \geq 0$ and $(x,y)\in TM$;
  \item[iii)] The \emph{metric tensor} with components
    \begin{eqnarray} g_{ij}(x,y)=\frac{1}{2}\frac{\partial^2 F^2}{\partial
        y^i\partial y^j} \ \textrm{\  has\ rank\ n\ on}\ T_0M. \label{gij}
    \end{eqnarray} \end{itemize}
\end{defn} The above conditions of Definition \ref{def:finsler} imply
that the metric tensor $g_{ij}$ of a Finsler function is positive
definite on $T_0M$, see \cite{Lovas07}.  Some relaxations of the above conditions,
which lead to the notion of \emph{conic pseudo-Finsler metric}, were
proposed in \cite{JS12}. See also \cite[\S 1.1.2, \S 1.2.1]{AIM93} for more discussions
about the regularity conditions and their relaxation for a Finsler function. In subsection \ref{nvrc} we will discuss an
example of a spray metrizable by such a conic pseudo-Finsler function.
 
A Finsler function is \emph{reducible to a Riemannian metric} if the
metric tensor $g_{ij}$ in formula \eqref{gij} does not depend on the
fibre coordinates $y$. If $g_{ij}(x)$ is a Riemannian metric on $M$, then
$F: TM \to \mathbb{R}$, $F(x,y)=\sqrt{g_{ij}(x)y^iy^j}$ is a Finsler function.  

 The regularity condition iii) of Definition
\ref{def:finsler} is equivalent to the fact that the Poincar\'e-Cartan
$2$-form of $F^2$, $\omega_{F^2}=-dd_JF^2$, is non-degenerate and hence
it is a symplectic structure. Therefore, the equation
\begin{eqnarray} i_Sdd_JF^2=-dF^2 \label{isddj}
\end{eqnarray} uniquely determine a vector field $S$ on $T_0M$, which is
called the \emph{geodesic spray} of the Finsler function.  In this
work we will use more frequently, the equation
$\mathcal{L}_Sd_JF^2=dF^2$, which is equivalent to equation \eqref{isddj}.
\begin{defn} A spray $S\in \mathfrak{X}(T_0M)$ is called \emph{Finsler metrizable} if there
  exists a Finsler function $F$ that satisfies the equation \eqref{isddj}.
\end{defn} Necessary and sufficient criteria for the Finsler
metrizability problem for a spray $S$ were formulated in
\cite{Muzsnay06} using the holonomy distribution
$\mathcal{H}_S$. Also, such necessary and sufficient conditions were
formulated in terms of a semi-basic $1$-form in \cite{BD09}. We will
use these conditions in the next section to discuss the Finsler
metrizability problem of a particular class of isotropic sprays.

\begin{defn} \label{scalarflag} Consider $F$ a Finsler
  function and $\Phi$ the Jacobi endomorphism of its geodesic spray
  $S$.  \begin{itemize} \item[i)] $F$ is said to be
  of \emph{scalar (constant) flag curvature} if there exists a scalar
  function (constant) $\kappa$ on
  $T_0M$, such that
\begin{eqnarray}
\Phi=\kappa\left(F^2 J - Fd_JF \otimes
  \mathbb{C}\right). \label{scalarphi} \end{eqnarray}
\item[ii)] $F$ is called an \emph{Einstein metric} if there exists a
  function $\lambda\in C^{\infty}(M)$ such that the Ricci scalar satisfies
  $R(x,y)=\lambda(x) F^2(x,y).$ \end{itemize}
\end{defn}
The notion of flag curvature extends to the Finslerian setting the
concept of sectional curvature from the Riemannian setting. 

\begin{rem} \label{gis} If a Finsler function $F$ is reducible to a Riemannian metric
$g$ on a manifold of dimension grater than or equal to three, and $S$ is its geodesic spray, then the following conditions are equivalent, see
\cite[\S 13.4]{Shen01}:
\begin{itemize}
\item[i)] $S$ is isotropic;
\item[ii)] $g$ is of scalar curvature;
\item[iii)] $g$ is of constant curvature.
\end{itemize}
\end{rem}
In the general Finslerian context, conditions ii) and iii) above are not
equivalent anymore. In the next section, we provide the necessary and
sufficient condition for an isotropic geodesic spray such that this
equivalence remains true. See the equivalence of the conditions of Theorem \ref{metrizables}.

\section{Finsler metrizable isotropic sprays}

In this section we use the necessary and sufficient conditions,
expressed in terms of a semi-basic $1$-form, which were formulated in
\cite[Thm. 5.4]{BD09}, to discuss the
Finsler metrizability problem for some classes of isotropic sprays.

\subsection{Sprays metrizable by Finsler functions of constant
curvature}

In the next theorem we provide the necessary and
sufficient conditions one has to check if we want to decide if a spray is
metrizable by a Finsler function of non-zero constant
curvature.

\begin{thm} \label{metrizablesc}
Consider $S$ a spray with non-vanishing Ricci curvature. The spray $S$
is metrizable by a Finsler function of non-zero constant flag
curvature if and only if its Jacobi endomorphism satisfies the following equations:
\begin{eqnarray}
& A) & \operatorname{rank} d d_J(\operatorname{Tr} \Phi) =2n \nonumber \\
& D_1)&  2(n-1)\Phi - 2(\operatorname{Tr} \Phi) J + d_J(\operatorname{Tr} \Phi)
\otimes \mathbb{C}=0; \label{fmiso}  \\
& D_2)& d_h(\operatorname{Tr} \Phi)=0.  \nonumber
\end{eqnarray} 
\end{thm}
\begin{proof}
Consider $S$ a spray with non-vanishing Ricci curvature. We assume
that its Jacobi endomorphism, $\Phi$, satisfies the
algebraic assumption $A)$  as well as the two tensorial equations
\eqref{fmiso}. Since $\Phi$ satisfies $D_1)$, it
follows that the the Jacobi endomorphism is given by formula
\eqref{isophi}, where $2(n-1)\alpha=(n-1) d_JR=d_J(\operatorname{Tr}
\Phi)$. Therefore the spray $S$ is isotropic
and satisfies the condition $d_J\alpha=0$. 

Due to condition $D_2)$, we have that
$S$ is Ricci constant and, as we have seen already, it follows that the spray $S$ is weakly Ricci constant. 

Using the fact that $2\alpha=d_JR$ we obtain
\begin{eqnarray} 2\mathcal{L}_S\alpha = \mathcal{L}_Sd_JR = d_{[S,
    J]}R + d_J\mathcal{L}_S R = d_vR = dR. \label{lsadr}
\end{eqnarray}
Within the assumption that the Ricci curvature does not vanish on
$T_0M$, we may consider the function $F>0$ such that
$F^2=\operatorname{sign}(R) R > 0$ on $T_0M$.  Since $dd_J (\operatorname{Tr} \Phi) =
(n-1) dd_JR = (n-1)dd_J F^2$, the assumption $A)$ assures that $F$ is a
Finsler function. The condition $2\alpha=d_JR$ reads now
$2\alpha=d_JF^2$ and using formula \eqref{lsadr} we obtain $\mathcal{L}_S
d_JF^2=dF^2$, which means that $S$ is the geodesic spray of the Finsler
function $F$. 

We replace $\operatorname{Tr} \Phi = (n-1)F^2=(n-1)R=(n-1)i_S\alpha$ and $d_J(\operatorname{Tr}
\Phi)  = 2(n-1) Fd_JF = 2(n-1)\alpha=(n-1) d_JR$ in the expression for
$\Phi$ and obtain formula \eqref{scalarphi} for $\kappa=\operatorname{sign}(R)$. It follows
that spray $S$ is Finsler metrizable by the Finsler function $F$ of
constant curvature $\kappa=\operatorname{sign}(R)$. 

Conversely, if the spray $S$ is Finsler metrizable by a Finsler function of non-zero constant flag
curvature then its Jacobi endomorphism is given by formula
\eqref{scalarphi}. It is a straightforward computation to see that
$\Phi$ satisfies all three conditions $A)$, $D_1)$ and $D_2)$ in \eqref{fmiso}.
\end{proof}

The algebraic condition $A)$ assures the regularity condition for the
sought after Finsler function. Condition $D_1)$ is equivalent to the fact that the spray $S$ is
isotropic, its Jacobi endomorphism is given by formula \eqref{isophi}, where
$2(n-1)\alpha = d_J(\operatorname{Tr} \Phi)$ and hence satisfies the
condition $d_J\alpha=0$. Condition $D_2)$ says that the spray $S$ is Ricci constant. 

Conditions $D_1)$ and $D_2)$ in \eqref{fmiso} are very useful to
decide whether or not a given spray is metrizable by a Finsler
function of non-zero constant curvature. However, there are metrizable
sprays by Finsler functions of non-constant flag curvature, and hence where conditions
$D_1)$ and $D_2)$ cannot be used. See the example in subsection
\ref{djano}. In Theorem \ref{metrizables} we will
strengthen the result of Theorem \ref{metrizablesc} by limiting our
discussion to the case where Finsler metrizability is equivalent to
the metrizability by a Finsler function of constant curvature. In
this discussion, we use the Finslerian version of Schur's Lemma,
\cite[Lemma 3.10.2]{BCS00}. Therefore, we will have to limit our
considerations to the case where 
$\operatorname{dim} M\geq 3$. The case $\operatorname{dim} M=2$ will
be treated separately.   

\subsection{dim $M\geq 3$}

Consider $S$ an isotropic spray, whose Jacobi endomorphism is given
by formula \eqref{isophi}.  In the next theorem we will see that the
condition $d_J\alpha=0$ is the best we can require to make sure that the three equivalent conditions in Remark
\ref{gis} are also true in the Finslerian context. We add two more
conditions, the last one is condition $D_2)$ in Theorem
\ref{metrizablesc}. 

\begin{thm} \label{metrizables}
Consider $S$ a spray of non-vanishing Ricci curvature. Then, the spray is isotropic, satisfies the algebraic
condition $A)$, and the condition
$d_J\alpha=0$, if and only if the following five conditions are equivalent:
\begin{itemize}
\item[i)] $S$ is Finsler metrizable; 
\item[ii)] $S$ is metrizable by a Finsler metric of non-vanishing scalar
  flag curvature; 
\item[iii)] $S$ is Finsler metrizable by an Einstein
metric; 
\item[iv)] $S$ is metrizable by a Finsler metric of non-zero constant 
  flag curvature; 
\item[v)] $S$ is Ricci constant. \end{itemize} 
\end{thm}
\begin{proof}
We assume that the spray is isotropic, satisfies the algebraic
condition $A)$, and the condition
$d_J\alpha=0$.  Within these assumptions, we will prove the following implications $i)\Longrightarrow ii)
\Longrightarrow iii) \Longrightarrow iv) \Longrightarrow v)
\Longrightarrow i)$.  For some of these implications we will not need
the condition $d_J\alpha=0$. More exactly, we will use it for the
implications $ ii) \Longrightarrow iii) $, $ iii) \Longrightarrow iv)$, and $v)
\Longrightarrow i)$.  Also, the assumption $\operatorname{dim} M\geq
3$ will be needed only for the implication $ iii) \Longrightarrow iv) $. 

In order to prove that condition i) implies condition ii) we assume
that the spray  $S$ is Finsler metrizable. 
We will prove that the corresponding Finsler space has scalar flag curvature. This
result coincides with Lemma 8.2.2 in \cite{Shen01}, where the proof
uses different (local) techniques. Our proof is based on the
differential calculus on $T_0M$ associated to a spray within the
Fr\"olicher-Nijenhuis formalism.

We consider $F$ a Finsler function that metricizes the spray $S$. It
follows that $S$ is the unique solution of the equation
$\mathcal{L}_Sd_JF^2=dF^2$. Therefore, the corresponding Helmholtz
conditions are satisfied. One of these Helmholtz conditions involves
the Jacobi endomorphism $\Phi$, and can be expressed using the
semi-basic $1$-form $\theta=d_JF^2$ as follows, see \cite[Thm. 4.1.]{BD09},
\begin{eqnarray}
0=d_{\Phi}\theta = d_{R J - \alpha \otimes \mathbb{C}}\theta =
R d_J\theta - d_{\alpha\otimes \mathbb{C}} \theta = -\alpha
\wedge \mathcal{L}_{\mathbb{C}}\theta = -\alpha \wedge \theta. \label{dphith}\end{eqnarray}
In the above formula we used $d_J\theta=d^2_JF^2=0$, since
$d^2_J=d_{[J,J]}=0$, and $\mathcal{L}_{\mathbb{C}}\theta = \theta$,
since $\theta$ is a $1$-homogeneous $1$-form. Helmholtz condition
\eqref{dphith} implies $\alpha\wedge d_JF^2=0$ and hence there exists
a function $\kappa \in C^{\infty}(T_0M)$ such that $\alpha =\kappa
d_JF^2/2=\kappa Fd_JF$. It follows that $R=i_S\alpha = \kappa F^2$
and the Jacobi endomorphism \eqref{isophi} can be written now as
in formula \eqref{scalarphi}, which shows that the Finsler space $(M,
F)$ has scalar flag curvature $\kappa$. 

In order to prove that condition $ii)$ implies condition $iii)$ we
make use of the differential assumption $d_J\alpha=0$. According to
Lemma \ref{lem:dja} it follows $d_J\alpha=0$ is equivalent to
$2\alpha=d_JR$. Using the fact that $R= \kappa F^2$ and $2\alpha
=\kappa  d_JF^2$ it follows 
$$2\alpha = d_JR = F^2 d_J\kappa + \kappa d_JF^2 = F^2 d_J\kappa +
2\alpha.$$ Above formulae imply that $d_J\kappa=0$ and hence the
scalar flag curvature $\kappa$ does not depend on the flagpole
$y$. It follows that $R(x,y)= \lambda(x) F^2(x,y)$, where
$\lambda(x)=\kappa(x)$. Hence, the Finsler Function $F$ is an
Eistein metric that metricizes the spray $S$.

In order to prove that condition $iii)$ implies condition $iv)$ we
will also make use of Schur's Lemma. We know that the spray $S$ is Finsler metrizable
by an Einstein metric $F$. Then, there exists a non-vanishing function $\lambda\in
C^{\infty}(M)$ such that $R(x,y)= \lambda(x) F^2(x,y)$. Since $S$ is
isotropic, it follows that its Jacobi endomorphism is given by formula
\begin{eqnarray*}
\Phi=\lambda F^2 J - \alpha\otimes \mathbb{C}. \end{eqnarray*}
Now we use the assumption that $d_J\alpha=0$, which by Lemma
\ref{lem:dja} implies that $2\alpha=d_JR=\lambda d_JF^2$. This implies
that the Jacobi endomorphism is given by formula \eqref{scalarphi},
where $\kappa(x)=\lambda(x)$ is the scalar flag curvature. Since $\operatorname{dim} M\geq
3$, we use  Schur's Lemma and obtain that $\kappa$ is a non-zero constant.
This implies that the spray $S$ is metrizable by the Finsler function
$F$ of constant flag curvature. 

The implication $iv) \Longrightarrow  v)$ is
straightforward. Consider $S$ the geodesic spray of Finsler metric
$F$ of non-zero constant flag curvature $\kappa$. It follows that the Ricci scalar is given by $R= \kappa F^2$,
where $\kappa$ is a constant. Therefore, $d_hR=\kappa d_hF^2=0$ since
$d_hF^2=0$.

For the last implication $v) \Longrightarrow  i)$, we use the first
implication of Theorem \ref{metrizablesc}. Since $S$ is isotropic and
satisfies $d_J\alpha=0$ it follows that the Jacobi endomorphism $\Phi$
satisfies equation $D_1)$ in \eqref{fmiso}. The fact that $S$ is Ricci
constant means that $\Phi$ satisfies also equation $D_2)$ in
\eqref{fmiso}.  By Theorem \ref{metrizablesc} we obtain that the spray
$S$ is Finsler metrizable. 

To conclude the proof, we have to show that if a spray $S$ satisfies
one of the five equivalent conditions of the theorem, then necessarily
$S$ is isotropic, satisfies the algebraic
condition $A)$, as well as the condition 
$d_J\alpha=0$. We assume that condition iv) is satisfied and hence
the spray $S$ is Finsler metrizable by a Finsler metric $F$ of
non-zero constant curvature $\kappa$. It follows that its Jacobi endomorphism is
  given by formula \eqref{scalarphi} and hence $\alpha=\kappa
  Fd_JF=\kappa d_J F^2$. Since $F$ is a Finsler function we obtain
  that $\operatorname{rank} dd_JF^2=2n$ and hence the algebraic condition $A)$ is
  satisfied. The condition $d_J\alpha=0$ is also satisfied. 
\end{proof}
If a Finsler function is reducible to a Riemannian metric, and $S$ is
its isotropic geodesic spray, it follows that $\alpha =\kappa d_JF^2$, where $d_J\kappa=0$, and hence we always have
$d_J\alpha=0$. This shows that the equivalence of 
conditions i), ii) and iv) of Theorem \ref{metrizables} generalize to
the Finslerian context the equivalence of
the three conditions of Remark \ref{gis}.

For Theorem \ref{metrizables}, the condition $\operatorname{dim} M\geq 3$ was very important, since it
allowed us to use the Finslerian version of Schur's Lemma. 

\subsection{dim $M$=2}

In this subsection we pay attention to the Finsler metrizability problem
on $2$-dimensional manifolds. It is known that any spray on
a $2$-dimensional manifold is isotropic. This result allows us to simplify
the two conditions \eqref{fmiso} of Theorem \ref{metrizablesc}.

\begin{thm} \label{metrizables2} Consider $S$ a spray of non-vanishing
  Ricci curvature that satisfies the algebraic
condition $A)$, for $n=2$.
\begin{itemize}
\item[i)] The spray $S$ is isotropic, which means that its Jacobi
  endomorphism  is given by formula \eqref{isophi}, where the
  semi-basic $1$-form $\alpha=\alpha_i dx^i$ has the components
\begin{eqnarray}
\alpha_1=\frac{R^2_2}{y^1}=\frac{-R^2_1}{y^2}, \quad
\alpha_2=\frac{-R^1_2}{y^1}=\frac{R^1_1}{y^2}. \label{alpha12} \end{eqnarray}
\item[ii)] The spray $S$ is metrizable by a Finsler function of non-zero
  constant curvature if and only it satisfies the following two
  conditions 
\begin{eqnarray}
d_J\alpha =0, \quad d_hR =0. \label{fm2} 
\end{eqnarray}
\end{itemize}
\end{thm}
\begin{proof}
i) Due to the homogeneity conditions of the spray we have that
$\Phi(S)=0$ and by formula \eqref{localphi} we obtain $R^i_jy^j=0$. It follows that $\Phi$
can be written as in formula \eqref{isophi}, where $R=R^1_1+R^2_2$ and
the semi-basic $1$-form $\alpha=\alpha_i dx^i$ has the components \eqref{alpha12}. 

Since any spray $S$ is isotropic we have that the Jacobi endomorphism
$\Phi$ satisfies equation $D_1)$ in \eqref{fmiso} if and only if it
satisfies first equation in \eqref{fm2}. This part is then a
consequence of Theorem \ref{metrizablesc}. 
\end{proof}

The first part of Theorem \ref{metrizables2} coincides with Lemma 8.1.10
in \cite{Shen01} and formulae \eqref{alpha12} coincide with formulae
(8.37) and (8.38) in \cite{Shen01}.  

The two conditions \eqref{fm2} can be written as follows
\begin{eqnarray}
\label{fm2local} & d_J\alpha  =  0 & \Leftrightarrow 2\alpha=d_JR  \Leftrightarrow \frac{\partial \alpha_1}{\partial y^2} = \frac{\partial
  \alpha_2}{\partial y^1} \Leftrightarrow 2\alpha_1=\frac{\partial
  R}{\partial y^1} \textrm{ and }  2\alpha_2=\frac{\partial
  R}{\partial y^2}   \\ \nonumber
& d_hR = 0 & \Leftrightarrow \frac{\delta R}{\delta x^1} = \frac{\delta
  R}{\delta x^2}=0, \textrm{ where } \frac{\delta}{\delta
  x^i}=h\left(\frac{\partial}{\partial x^i}\right) =
\frac{\partial}{\partial x^i} - \frac{\partial G^j}{\partial y^i} \frac{\partial}{\partial y^j}.
\end{eqnarray}
We will use the above conditions in the next section to test whether or
not sprays on a $2$-dimensional manifold are metrizable by Finsler
functions of constant curvature. 

\section{Examples}

In this section we provide examples to show the consistency of the
conditions we discussed so far.

\subsection{The case $d_J\alpha=0$} \label{djao}
It is well known that the class of isotropic sprays is invariant under
projective deformations of sprays. We start with a Finsler
metrizable isotropic spray, $S_0$, 
which satisfies the condition $d_J\alpha=0$. We will
study projective deformations of $S_0$, which also satisfy the condition
$d_J\alpha=0$. Within this projective deformations, we will seek for
those which do not preserve the condition of being Ricci
constant. Using Theorem \ref{metrizables}, this will lead us to    
examples of non-Finsler metrizable isotropic sprays. 

Let $S_0$ be the geodesic spray of a Finsler function $F_0$, which has
constant flag curvature $\kappa_0$. Denote by $h_0$ the horizontal
projector induced by the spray $S_0$ and $\nabla_0$ the corresponding dynamical
covariant derivative. The spray $S_0$ is isotropic, its
Jacobi endomorphism is given by 
\begin{eqnarray*}
\Phi_0 = \kappa_0(F_0^2 J - F_0 d_J F_0 \otimes
\mathbb{C}). \end{eqnarray*}
Consider the projectively equivalent spray $S=S_0 - 2 P
\mathbb{C}$, where $P \in C^{\infty}(T_0M)$ is a $1$-homogeneous
function. According to formulae (4.8) in \cite[Prop. 4.4]{BM12}, the
horizontal projector and the  Jacobi endomorphism of the spray $S$ are
given by  
\begin{eqnarray*}
h & = & h_0 - 2 (PJ + d_JP\otimes \mathbb{C}), \\
\Phi & = & \Phi_0  + (P^2 - S_0(P))  J + (2d_{h_0}P - Pd_JP - \nabla_0
d_JP) \otimes \mathbb{C}. \end{eqnarray*}
The spray $S$ is also isotropic. We will study now when the projective
deformations preserve the condition $d_J\alpha=0$,  where the $1$-form $\alpha$
is given by $\alpha=  \kappa_0 F_0 d_J F_0 - 2d_{h_0}P + Pd_JP + \nabla_0
d_JP $. The $2$-form $d_J\alpha$
has been computed in the proof of Proposition 4.4 in \cite{BM12} and
it is given by 
\begin{eqnarray}
d_J\alpha = 3d_Jd_{h_0}P = -3d_{h_0}d_JP. \label{djap}
\end{eqnarray} Therefore $d_J\alpha=0$ if and only if $d_{h_0}P=d_J
g$, for some function $g\in C^{\infty}(T_0M)$. For $g=P^2/2$ the function $P$ is called a \emph{Funk function}, see
\cite[Def. 12.1.4]{Shen01}. It has been shown in \cite[Prop. 12.1.3]{Shen01}
that projective deformations by a Funk function preserve the Jacobi
endomorphism. 

In order to simplify some of the calculations we assume that the
function $P$ satisfies $d_{h_0}P=0$ and hence $d_J\alpha=0$. It
follows that $S_0(P)=0$ and using the commutation formula (4.11) in \cite{BM12} we have 
$\nabla_0 d_JP=d_J \nabla_0 P - d_{h_0}P=0$. Therefore, the Jacobi
endomorphism $\Phi$ of the spray $S$ is given by
\begin{eqnarray} 
\Phi = (\kappa_0 F^2_0 + P^2) J - (\kappa_0 F_0d_{J} F_0 + Pd_JP) \otimes
\mathbb{C}. \label{phip} \end{eqnarray}
It follows that the Ricci scalar is given by 
$$ R=\kappa_0F_0^2 + P^2. $$ 

We check now the last condition of Theorem \ref{metrizables}. Using
the assumption that $d_{h_0}P=0$ it follows that $d_{h_0}R=0$ and
hence we have 
\begin{eqnarray}
d_hR =  - 2 d_{PJ + d_JP \otimes \mathbb{C}} R = -2 (Pd_JR + 2Rd_JP).  \label{dhrp}
\end{eqnarray}
We can choose a function $P$ such that $d_hR\neq 0$, which means
that the spray $S$ is not Ricci constant. According to
Theorem \ref{metrizables} the spray $S$ is not Finsler
metrizable. Indeed, we can take $P=\lambda F_0$, where $\lambda$ is a
constant. If we replace this in
formula \eqref{dhrp}, we obtain 
\begin{eqnarray}
d_hR =  -4\lambda (\kappa_0 + \lambda^2) d_JF^2_0.  \label{dhrl}
\end{eqnarray}
In this case we have that the spray $S$ is Ricci constant, and hence it
is Finsler metrizable, if and only if either $\lambda=0$ or
$\lambda^2+\kappa_0=0$. See also Theorem 5.1 in \cite{BM12}. When
the spray $S_0$ is projectively flat, we can view this as an alternative proof
of Theorem 1.2 in \cite{Yang11}. 

\subsection{The case $d_J\alpha \neq 0$.} \label{djano}
We present now an example of a Finsler metrizable, isotropic spray,
which does not satisfy the condition $d_J\alpha =0$. This also shows that the assumption $d_J\alpha =0$, which we made in Theorem
\ref{metrizables}, is the best assumption we could consider in order to have
the equivalence of the five conditions. 

We will use the following Randers metric studied by Shen, see Example
11.2 in \cite{Shen04}. Consider a domain $M \subset
\mathbb{R}^n$, where $\Delta(x) = 1 - |a|^2 |x|^4 >0$. Denote by
$\beta(x,y)=2<a,x> <x, y> - |x|^2 <a,y> $. The Finsler
function $F: M\times \mathbb{R}^n \to \mathbb{R}$, given by 
\begin{eqnarray*}
F(x,y)= \frac{\sqrt{\beta^2(x,y) + \Delta(x) |y|^2} +\beta(x,y)}{\Delta(x)}
\end{eqnarray*}
has scalar flag curvature given by 
\begin{eqnarray*}
\kappa(x,y) = 3\frac{<a,y>}{F} + 3<a,x>^2 -2|a|^2|x|^2. 
\end{eqnarray*}
The geodesic spray $S$ of the Finsler function $F$ is isotropic and
the $1$-form $\alpha$ in formula \eqref{isophi} is given by 
$ \alpha = \kappa Fd_JF $ and hence 
$$ d_J\alpha = F d_J\kappa \wedge d_J F. $$ 
The scalar flag curvature $\kappa$ is $0$-homogeneous and
therefore $0=\mathbb{C}(\kappa)=i_Sd_J\kappa$. Moreover the flag curvature
$\kappa$ depends on the flagpole $y$, which means that $d_J\kappa\neq 0$. Therefore,
$$ i_Sd_J\alpha = Fi_Sd_J\kappa d_JF - Fi_Sd_JF d_J\kappa = -
F^2d_J\kappa \neq 0,$$ 
and this implies that $d_J\alpha \neq 0$. Therefore, the spray $S$ is
Finsler metrizable and isotropic. However, the five conditions in
Theorem \ref{metrizables} are not equivalent and this is due to the
fact that $d_J\alpha \neq 0$.

\subsection{Two-dimensional examples}
We consider now some examples of sprays on a two-dimensional
manifold and use the conditions \eqref{fm2} to test if they are
metrizable by a Finsler function of constant curvature.

\subsubsection{The Poincar\'e model and the Finslerian Poincar\'e disk} Consider the geodesic equations of the Poincar\'e half plane
$M=\{(x^1, x^2)\in \mathbb{R}^2, x^2>0\}$:
\begin{eqnarray*}
\frac{d^2x^1}{dt^2} &-& \frac{2}{x^2} \frac{dx^1}{dt}\frac{dx^2}{dt} =
0, \\
\frac{d^2 x^2}{dt^2} &+&\frac{1}{x^2}\left(
  \left(\frac{dx^1}{dt}\right)^2 -  \left(\frac{dx^2}{dt}\right)^2
\right) = 0.
\end{eqnarray*}
The above system of second order ordinary differential equations
determines a spray $S\in \mathfrak{X}(TM)$. For this spray $S$, the local components \eqref{localphi} of the
Jacobi endomorphism  are given by
\begin{eqnarray*}
R^1_1= -\frac{(y^2)^2}{(x^2)^2}, \quad R^1_2 = R^2_1 =
\frac{y^1y^2}{(x^2)^2}, \quad R^2_2=
-\frac{(y^2)^2}{(x^2)^2}. \end{eqnarray*}
According to  first part of Theorem \ref{metrizables2}, it follows that
the spray $S$ is isotropic, with the two components of the semi-basic
$1$-form $\alpha$ given by formula \eqref{alpha12}:
\begin{eqnarray*}
\alpha_1=\frac{R^2_2}{y^1}=-\frac{y^1}{(x^2)^2}, \quad
\alpha_2=\frac{R^1_1}{y^2}=-\frac{y^2}{(x^2)^2}. \end{eqnarray*}
It is very easy to check that $d_J\alpha=0$ and hence the first
condition \eqref{fm2} is satisfied. The Ricci scalar of the spray
is given by 
\begin{eqnarray}
R=R^1_1+R^2_2=-\frac{1}{(x^2)^2}\{(y^1)^2+(y^2)^2\}.\label{rh2}\end{eqnarray}
 It follows that $d_hR=0$ and hence the second condition \eqref{fm2}
 is satisfied. By the second part of Theorem \ref{metrizables2} we
 obtain that $S$ is metrizable by a Finsler function of constant
 negative sectional curvature. From formula \eqref{rh2} it follows that, up to a
 multiplicative constant, the Finsler function $F$ is given by 
\begin{eqnarray*}
F(x,y)=\frac{1}{x^2}\sqrt{(y^1)^2+(y^2)^2}. \label{fh2}\end{eqnarray*} 
The above Finsler function is reducible to a Riemannian metric,
which we can recognize to be the Poincar\'e metric of the upper
half plane. 

Although the previous example is Riemannian, one can
modify it to obtain a Finslerian
 one. Consider the disk $M=\{(x^1, x^2)\in \mathbb{R}^2,
 (x^1)^2+(x^2)^2<4\}$  with the following Finsler function $F:TM \to [0,
 +\infty)$, known as the Finslerian Poincar\'e disk \cite{BCS00, BR04},
\begin{eqnarray*}
F=4\frac{\sqrt{(y^1)^2+(y^2)^2}}{4-r^2} + 16\frac{x^1y^2+x^2y^1}{(4-r^2)(4+r^2)},
\end{eqnarray*} where $r^2=(x^1)^2+(x^2)^2$. In the first term of the
right hand side of the above formula we can recognize the Poincar\'e
metric on the disk $M$. The above Finsler
function has constant flag curvature $\kappa=-1/4$, and its geodesic
spray satisfies all the assumptions of Theorem \ref{metrizables2}. 

\subsubsection{Non-constant Ricci scalar} We consider an example
proposed by Bao and Robles in \cite{BR04} of a two-dimensional spray,
which is metrizable by a Finsler function of non-constant Ricci
scalar. For this example, the first condition \eqref{fm2}  is
satisfied, but the second condition \eqref{fm2} it is not. 

Consider the portion of the elliptic paraboloid $M=\{(x^1, x^1, x^3)\in \mathbb{R}^3, x^3=(x^1)^2+(x^2)^2,
(x^1)^2+(x^2)^2<1\}$. We denote $\Delta(x^1, x^1)=1-(x^1)^2 -
(x^2)^2>0$. The Randers metric $F:TM\to [0, +\infty)$,
$F=\alpha+\beta$, where   
\begin{eqnarray*}
\alpha &=& \frac{\sqrt{(x^1y^2-x^2y^1)^2 +
    \left((1+4(x^1)^2)(y^1)^2 + 8x^1x^2 y^1y^2 + (1+4(x^2)^2)(y^2)^2
    \right)\Delta}}{\Delta}, \\ \beta &=&  \frac{x^2y^1-x^1y^2}{\Delta(x^1, x^2)}, \end{eqnarray*}
has scalar flag curvature
\begin{eqnarray*}
\kappa (x^1, x^2) = \frac{4}{(1+4(x^1)^2+4(x^2)^2)^2}. \end{eqnarray*}
Let $S$ be the geodesic spray of the Finsler function $F$ and $h$ the
horizontal projector. It follows
that $S$ is isotropic, its Jacobi endomorphism is given by formula
\eqref{isophi}, where 
\begin{eqnarray*} R=\kappa F^2, \quad \alpha =\kappa Fd_JF.\end{eqnarray*}
For the two-dimensional spray $S$,  we check now the conditions of the
Theorem \ref{metrizables2}. The Ricci scalar $R$ does no vanish in
$T_0M$. Since $d_J\kappa=0$ it follows that $d_J\alpha=0$ and hence
first condition \eqref{fm2} is satisfied. However, since
$d_h\kappa\neq 0$ and $d_hF=0$ it follows that $d_hR\neq 0$, which
shows that the spray $S$ is not Ricci constant, which means that
second condition \eqref{fm2} is not satisfied.

\subsection{The non-vanishing condition of the Ricci curvature} \label{nvrc}

In this subsection we discuss the non-vanishing condition of the Ricci
curvature, which we assumed in our results. We will explain that although the Ricci
curvature might vanish, if it is not identically zero, then the 
conditions of metrizability in Theorems \ref{metrizablesc}, \ref{metrizables} and
\ref{metrizables2}  are still necessary and sufficient, but the
metrizability to consider has to be with respect to a conic
pseudo-Finsler function, as defined in \cite{JS12}. 

Consider the following affine spray on some open domain $M\subset \mathbb{R}^2$, which is
Example 8.2.4 from \cite{Shen01}: 
\begin{eqnarray*}
S=y^1\frac{\partial}{\partial x^1} + y^2\frac{\partial}{\partial x^2}
- \phi(x^1, x^2) (y^1)^2 \frac{\partial}{\partial y^1} - \psi(x^1,
x^2) (y^2)^2 \frac{\partial}{\partial y^2}. \end{eqnarray*}  
For this spray $S$, the local components \eqref{localphi} of the
Jacobi endomorphism  are given by
\begin{eqnarray*}
R^1_1= -\phi_{x^2}y^1y^2, \quad R^1_2 = \phi_{x^2}(y^1)^2, \quad  R^2_1 =
\psi_{x^1} (y^2)^2, \quad R^2_2=
-\psi_{x^1}y^1y^2 . \end{eqnarray*}
The spray $S$ is isotropic and the two components of the semi-basic
$1$-form $\alpha$ are given by formula \eqref{alpha12}:
\begin{eqnarray*}
\alpha_1=\frac{R^2_2}{y^1}=-\psi_{x^1} y^2, \quad
\alpha_2=\frac{R^1_1}{y^2}=-\phi_{x^2} y^1. \end{eqnarray*} 
In view of the local formulae \eqref{fm2local}, the condition $d_J\alpha=0$ is satisfied if and only if
\begin{eqnarray*} \frac{\partial \alpha_1}{\partial y^2}=
  \frac{\partial \alpha_2}{\partial y^1}, \textrm{ which \ is
   \ equivalent \ to: \  } \psi_{x^1}= \phi_{x^2}. \label{phipsi} \end{eqnarray*}
Next, we assume that the above condition is satisfied. Within this
assumption, it follows that the Ricci scalar is given by 
\begin{eqnarray*} R=-y^1y^2(\phi_{x^2} + \psi_{x^1}) = -2y^1y^2 \phi_{x^2}= -2y^1y^2
  \psi_{x^1}. \end{eqnarray*}
To test the second condition \eqref{fm2}, we use the local expression
\eqref{fm2local}:
\begin{eqnarray*} 
\frac{\delta R}{\delta x^1} = 2y^1y^2 (\phi\phi_{x^2} -
\phi_{x^1x^2}), \quad \frac{\delta R}{\delta x^2} = 2y^1y^2 (\psi\psi_{x^1} -
\psi_{x^1x^2}) 
\end{eqnarray*}
For functions $\phi,  \psi \in C^{\infty}(M)$ consider the following
system of partial differential equations:
\begin{eqnarray}
\phi_{x^2}=\psi_{x^1}, \quad \phi_{x^1x^2}- \phi \phi_{x^2} =0, \quad
\psi_{x^1x^2}- \psi \psi_{x^1} =0. \label{phipsi} \end{eqnarray}
For example on $M=\{(x^1, x^2)\in \mathbb{R}^2, x^1+x^2>0\}$ we can
consider the functions 
\begin{eqnarray} \phi(x^1, x^2)=\psi(x^1,
  x^2)=-\frac{2}{x^1+x^2}, \label{exphi} \end{eqnarray} which satisfy
the above system of partial differential equations. For this choice of
functions, the spray 
\begin{eqnarray*}
S=y^1\frac{\partial}{\partial x^1} + y^2\frac{\partial}{\partial x^2}
+ \frac{2(y^1)^2}{x^1+x^2}\frac{\partial}{\partial y^1} + \frac{2(y^2)^2}{x^1+x^2}\frac{\partial}{\partial y^2} \end{eqnarray*}  
satisfies the two conditions \eqref{fm2} of Theorem
\ref{metrizables2}. Since the Ricci curvature of the spray $S$ is given
by 
\begin{eqnarray*} R=-\frac{4y^1y^2}{(x^1+x^2)^2}, \end{eqnarray*}
it follows that the spray $S$ is metrizable by a conic pseudo-Finsler
function, see \cite{JS12}, $F: A=\{(x^1,x^2, y^1, y^2)\in TM, y^1y^2>0\} \subset TM \to [0, +\infty)$, where  
\begin{eqnarray*} F(x,y)=\frac{4y^1y^2}{(x^1+x^2)^2}, \end{eqnarray*}
has constant sectional curvature $\kappa=-1$.

\subsection*{Acknowledgments} The work of I.B. has been supported by the
Romanian National Authority for Scientific Research,
CNCS UEFISCDI, project number
PN-II-ID-PCE-2012-4-0131. The work of
Z.M. has been supported by the Hungarian Scientific Research Fund (OTKA) Grant K67617.

\end{document}